\documentclass{amsart}
\date{April 3, 2008}
\usepackage{amsmath,amssymb,enumerate}

\usepackage{epsfig}

\usepackage{psfrag}
\usepackage{graphicx}

\newcommand{\tq}{\, \big| \, }
\newtheorem{theorem}{\rm\bf Theorem}[section]
\newtheorem{proposition}{\rm\bf Proposition}[section]
\newtheorem{lemma}[proposition]{\rm\bf Lemma}
\newtheorem{corollary}[proposition]{\rm\bf Corollary}

\numberwithin{equation}{section}

\theoremstyle{definition}
\newtheorem{definition}[proposition]{\rm\bf Definition}

\theoremstyle{remark}
\newtheorem{remark}[proposition]{\rm\bf Remark}

\newtheorem{example}[proposition]{\rm\bf Example} 
\newtheorem{examples}[proposition]{\rm\bf Examples}

\def\interieur#1{\mathord{\mathop{\kern 0pt #1}\limits^\circ}}

\title{Weak Finsler Strutures and the Funk Metric}

\author{Athanase Papadopoulos}
\address{A. Papadopoulos, Institut de Recherche Math{\'e}matique Avanc\'ee,
Universit{\'e} Louis Pasteur and CNRS,
7 rue Ren\'e Descartes,
 67084 Strasbourg Cedex - France} \email{papadopoulos@math.u-strasbg.fr}

\author{Marc Troyanov}
\address{M. Troyanov, Section de Math{\'e}matiques,  \'Ecole Polytechnique F{\'e}d\'erale de
Lausanne, 1015 Lausanne - Switzerland}
\email{marc.troyanov@epfl.ch}

\begin{document}

\begin{abstract}  
We discuss general   notions of metrics and of Finsler structures which we call \emph{weak metrics} and \emph{weak Finsler structures}.  Any convex domain carries a canonical weak Finsler structure, which we call its \emph{tautological weak Finsler structure}. We compute distances in the tautological weak Finsler structure of a domain and we show that these are given by
the so-called \emph{Funk weak metric}. We conclude the paper with a discussion of  geodesics,  of metric balls and of convexity properties of the
Funk weak metric.

\bigskip

\end{abstract}

\maketitle
\tableofcontents

\section{Introduction}\label{intro}

A \emph{weak metric} on a set is a function defined on pairs of points in that set which is nonnegative, which can take the value $\infty$, which vanishes when the two points coincide and which satisfies the triangle inequality. Compared to an ordinary metric, a weak metric can thus degenerate and take infinite values. Besides, it need not be symmetric. This is  a very general notion which turns out to be useful in various situations. The terminology ``weak metric" is due to Ribeiro \cite{Rib}, but the notion can be at least traced back to the work of Hausdorff  (see \cite{Hausdorff}). In the paper  \cite{PT1}, a number of natural weak metrics are discussed. In the present paper, we are mostly interested in a class of weak metrics that is related to convex geometry and to a general notion of Finsler structures on manifolds. 

\medskip

A basic construction in convex geometry is the notion of \emph{Minkowski norm}, which associates to any convex set containing the origin in a vector space $V$ a translation-invariant homogenous weak metric on $V$. Finsler geometry is an extension  of this construction to an arbitrary manifold. We define a \emph{weak Finsler strucure} on a differentiable manifold to be a field of convex sets on that manifold. More precisely, a weak Finsler strucure is a subset of the tangent space of the manifold whose intersection with each fiber is an convex set containing the origin. The  Minkowski norm in each tangent space of a manifold endowed with a weak Finsler structure gives rise to a function defined on the total space of the tangent bundle. We call this function the \emph{Lagrangian} of the weak Finsler structure. Integrating this Lagrangian on piecewise smooth curves in the manifold defines a length structure and thus a notion of distance on the manifold. This distance is generally a weak metric.

\medskip

A case of special interest is when the manifold is a convex domain $\Omega$ in $\mathbb{R}^n$ and when the weak Finsler structure is obtained by replicating at each point of $\Omega$ the domain $\Omega$ itself. We call this the \emph{tautological} weak Finsler structure, and we study some of its basic properties in the present paper. More precisely, we first give a formula for the distance between two points. It turns out that this distance coincides with the metric introduced by P. Funk in  \cite{Funk}. We then study the geometry of balls and the geodesics in the Funk weak metric.

 Modern references on Finsler geometry include \cite{CS2005},  \cite{BBC}, \cite{BCS},  and  \cite{Paiva-Duran}. One of Herbert Busemann's major ideas, expressed in 
 \cite{Busemann1942},  
\cite{Busemann1944}, \cite{Busemann1955} and
\cite{Busemann1970}   is that Finsler geometry should be developed without local coordinates and without the use of differential calculus. This paper brings some results in that direction.

\section{Preliminaries on convex geometry}\label{s:1}

In this section, we recall  a few notions in convex geometry that will be used in the sequel.

 Given a convex subset $\Omega$ of $\mathbb{R}^n$, we shall denote $\overline{\Omega}$ its closure, $ \stackrel{o}\Omega$ its interior, and $\partial \Omega = \overline{\Omega}\setminus \stackrel{o}\Omega$  its boundary.

Let $\Omega\subset \mathbb{R}^n$ be a (not necessarily open) convex set and let $x$ be a point in $\Omega$.
\begin{definition}\label{def:radial}
The \emph{radial function of $\Omega$ with respect to $x$} is the function $r_{\Omega,x}:\mathbb{R}^n\to \mathbb{R}_+\cup\{\infty\}$ defined by
\[
r_{\Omega,x}(\xi)=\sup\{t\in\mathbb{R}\ \vert \ (x+t\xi)\in \Omega\}.\]
\end{definition}

\begin{definition}\label{def:Minkowski} The \emph{Minkowski function of $\Omega$ with respect to $x$} is the function  $p_{\Omega,x}:\mathbb{R}^n\to \mathbb{R}_+\cup\{\infty\}$ defined by
\[
p_{\Omega,x}(\xi) = \frac{1}{r_{\Omega,x}(\xi)}.\]
\end{definition}

Classically, the Minkowski function is associated to an open convex subset $\Omega$ of $\mathbb{R}^n$  containing the origin $0$, and taking $x=0$. This function is sometimes called the \emph{Minkowski weak norm} of the convex  (see e.g. \cite{Eggleston},  \cite{Minkowski},  \cite{Thompson} and  \cite{Webster}).

We  also recall that for any convex set $\Omega$ in $\mathbb{R}^n$, there exists a well-defined smallest affine subspace $L$ of  $\mathbb{R}^n$ containing $\Omega$, and that the intersection of $\Omega$ with $L$ has nonempty interior in $L$. We denote by  $\mathrm{RelInt}(\Omega)$ this interior, called the \emph{relative interior of the convex set $\Omega$}.

The following proposition collects a few basic properties of the Minkowski function. In particular, Property  (\ref{pro:8}) tells us that we can reconstruct the relative interior of $\Omega$ from the Minkowski function of $\Omega$ at any point. The proofs are easy.
\begin{proposition}\label{prop:min}
Let $\Omega$ be a convex subset of $\mathbb{R}^n$. For every $x$ in $\Omega$ and for every $\xi$ and $\eta$ in $\mathbb{R}^n$, we have

\begin{enumerate}
\item $p_{\Omega,x}(\xi)= \inf\{t\geq 0\ \vert \ \xi\in t(\Omega-x)\}$. (Here, $\Omega-x$ denotes the Minkowski sum of $\Omega$ and $-x$.)
\item If the ray $\{x+t\xi\ \vert \ t\geq 0\}$ is contained in $\Omega$, then $p_{\Omega,x}(\xi)= 0$. 
\item $p_{\Omega,x}(\lambda\xi)= \lambda p_{\Omega,x}(\xi)$ for $\lambda\geq 0$.
\item $p_{\Omega,x}(\xi+\eta)\leq p_{\Omega,x}(\xi) + p_{\Omega,x}(\eta)$.
\item The Minkowski function $p_{\Omega,x}$ is convex.
\item  If $x$ is in $\stackrel{o}\Omega$, then $p_{{\Omega},x}$ is continuous.
\item If ${\Omega}$ is closed, then ${\Omega}=\{y\in\mathbb{R}^n \ \vert \ y=x+\xi, p_{{\Omega},x}(\xi)\leq 1\}$.
\item \label{pro:8} $\mathrm{RelInt}(\Omega) = \{y=x+\xi \vert \ p_{{\Omega},x}(\xi) < 1\}$.
\item \label{pp} If $\Omega_1 = \mathrm{RelInt}(\Omega)$, then $p_{{\Omega_1},x} = p_{{\Omega},x}$.

\end{enumerate}
\end{proposition}

In some cases, we can give explicit formulas for the Minkowski  function $p_{{\Omega},x}$. For instance, the Minkowski function of the closed  ball $B=B(0,R)$ in $\mathbb{R}^n$ of radius $R$ and center $0$ with respect to any point $x$ in $B$ is given by 
\[p_{B,x}(\xi) = \frac{\sqrt{\langle \xi,x\rangle^2+ (R^{2}-\vert x\vert^2)\vert \xi\vert^2}+\langle \xi,x\rangle}{(R^{2}-\vert x\vert^2)}.
\]
The Minkowski function of a half-space $ H= \{ x \in \mathbb{R}^n \tq \langle \nu , x \rangle  \leq s\}$,
where $\nu$ is a vector in $\mathbb{R}^n$ (which is orthogonal to the hyperplane bounding $H$) and where $s$ is a real number,
with respect to a point $x$ in $H$, is given by
\[  p_{H,x}(\xi) = 
  \max\left( \frac{\langle \nu ,\xi  \rangle}{s-\langle \nu , x \rangle},0\right).
\]
 
 We shall use this formula later on in this paper.
 We also recall the following:
 
\begin{definition}[Support hyperplane] Let $\Omega$ be  a nonempty subset of $\mathbb{R}^n$. An affine  hyperplane $A$ in $\mathbb{R}^n$ is called a \emph{support hyperplane} for ${\Omega}$ if  the relative interior of ${\Omega}$ is contained in one of the two closed half-spaces bounded by $A$ and if $\overline{\Omega}\cap A\not=\emptyset$.

If $A$ is a support hyperplane for ${\Omega}$ and if $x$ is a point in $\overline{{\Omega}}\cap A$, then $A$ is called a \emph{support hyperplane for ${\Omega}$ at $x$.} When $\Omega \subset \mathbb{R}^2$, then $A$ is called a {\it support line}.
\end{definition}

Suppose now that $\Omega$ a convex subset of $\mathbb{R}^n$.
It is known that any point on the boundary of ${\Omega}$ is contained in at least one of its support hyperplanes (see e.g. \cite{Eggleston} p. 20). 
The intersection of ${\Omega}$ with any of its support hyperplanes is a convex set which is nonempty if ${\Omega}$ is closed. This intersection is not always reduced to a point.

We recall the notion of a strictly convex subset in $\mathbb{R}^n$, and before that we note the following classical proposition:
    
 \begin{proposition}\label{prop: strictly-convex}
Let $\Omega$ be an open convex subset of $\mathbb{R}^n$. Then, the following are equivalent:
\begin{enumerate}
\item $\partial\Omega$ does not contain any nonempty open affine segment;

\item each support hyperplane of $\Omega$ intersects $\partial\Omega$ in exactly one point;

\item support hyperplanes at distinct points of $\partial\Omega$ are distinct;

\item any linear function on $\mathbb{R}^n$ has exactly one maximum on $\partial\Omega$.
\end{enumerate} 
\end{proposition}

 \begin{definition}[Strictly convex subset]
Let $\Omega$ be an open convex subset of $\mathbb{R}^n$. Then, $\Omega$ is said to be \emph{ strictly convex} if one (or, equivalently, all) the properties of Proposition \ref{prop: strictly-convex} are satisfied.
\end{definition}

% ________________________   
\section{The notion of weak metric}

 \begin{definition}
 A \emph{weak metric} on a set $X$ is a function $\delta : X\times X\to \mathbb{R}_+ \cup\{\infty\}$ satisfying
 \begin{enumerate}
 \item $\delta(x,x)=0$ for all $x$ in $X$;
 \item $\delta(x,z)\leq \delta(x,y)+\delta(y,z)$ for all $x$, $y$ and $z$ in $X$.
 \end{enumerate}
 \end{definition}
 
We say that such a weak metric $\delta$ is \emph{symmetric} if $\delta(x,y)=\delta(y,x)$ for all $x$ and $y$ in $X$; that it is \emph{finite} if $\delta(x,y)<\infty$ for every $x$ and $y$ in $X$; that $\delta$ is \emph{strongly separating} if we have the equivalence
\[\min(\delta(x,y),\delta(y,x))=0\iff x=y;\]
and that $\delta$ is  \emph{weakly separating} if we have the equivalence
\[\max(\delta(x,y),\delta(y,x))=0\iff x=y.\]

We recall that the notion of weak metric already appears in the work of Hausdorff (cf. \cite{Hausdorff}, in which Hausdorff defines asymmetric distances on various sets of subsets of a metric space).

\begin{definition}[Geodesic] Let  $(X,\delta)$ be a weak metric space and let $I\subset\mathbb{R}$ be an interval. We say that a map $\gamma:I\to X$ is \emph{geodesic} if for every $t_1$, $t_2$ and $t_3$ in $I$ satisfying $t_1\leq t_2\leq t_3$ we have \[\delta(\gamma(t_1),\gamma(t_2))+ \delta(\gamma(t_2),\gamma(t_3))=\delta(\gamma(t_1),\gamma(t_3)).\]
\end{definition}

Weak metrics were extensively studied by Busemann, cf. \cite{Busemann1942},  \cite{Busemann1944}, \cite{Busemann1955} \& \cite{Busemann1970}. A basic example of a weak metric defined on a convex set in $\mathbb{R}^n$ is the following:

\begin{example} \label{ex:Minkowski} Let $\Omega\subset\mathbb{R}^n$ be a convex set such that $0\in\overline{\Omega}$  and let  $p(\xi) = p_{\Omega,0}(\xi)  =\inf\{t>0\ \vert \  \xi\in t \, \Omega\}$ be the Minkowski weak norm centered at $0$ of $\Omega$.
Then, the function $\delta:\mathbb{R}^n\times \mathbb{R}^n\to \mathbb{R}_+\cup\{\infty\}$ defined by 
\[\delta(x,y)=p(y-x)\]
is a weak metric on $\mathbb{R}^n$. For this weak metric, we have the following equivalences:
\begin{enumerate}
\item $\delta$ is finite $\iff$ $0\in \stackrel{o}\Omega$;
\item $\delta$ is symmetric $\iff$ $\Omega=-\Omega$;
\item $\delta$ is strongly separating $\iff$ $\Omega$ is bounded;
\item $\delta$ is weakly separating $\iff$ $\Omega$ does not contain any Euclidean line.
\end{enumerate}
\end{example}
 
 The weak metric on $\mathbb{R}^n$  defined in Example \ref{ex:Minkowski} is called  the  \emph{Minkowski weak metric associated to $\Omega$}. The associated weak metric space $(\mathbb{R}^n,\delta)$ is called a \emph{weak Minkowski space}.

\section{Weak length spaces}
 
Let $X$ be a set and let $\Gamma$ be a groupoid of paths in $X$.  Concatenation of paths is denoted by the symbol $*$.
The inverse $\gamma^{-1}$ of a path $\gamma:[a,b]\to X$ is the path obtained by pre-composing $\gamma$ with the unique affine sense-reversing self-homeomorphism of  $[a,b]$.

\begin{definition}
A \emph{weak length structure} on $(X,\Gamma)$ is a function $\ell:\Gamma\to\mathbb{R}_+\cup\{\infty\}$ 
which satisfies the following properties:
\begin{enumerate}
\item \emph{Invariance under reparametrization}:  if $[a,b]$ and $[c,d]$ are intervals of $\mathbb{R}$, if $\gamma:[a,b]\to X$ is a path in $X$ that belongs to $\Gamma$ and if $f:[c,d]\to [a,b]$ is an increasing homeomorphism such that $\gamma\circ f\in\Gamma$, then $\ell(\gamma)=\ell(\gamma\circ f)$.
\item \emph{Additivity}: for every  $\gamma_1$ and $\gamma_2$ in $\Gamma$, we have $\ell(\gamma_1 *\gamma_2)=\ell(\gamma_1) + \ell(\gamma_2)$.
 \end{enumerate}
\end{definition}

A weak length structure $\Gamma$ is said to be \emph{reversible} if for every $\gamma$ in $\Gamma$,  $\gamma^{-1}$ is also in $\Gamma$ and we have $\ell(\gamma^{-1})=\ell(\gamma)$.

A weak length structure $\Gamma$ is said to be \emph{separating} if we have the   equivalence:
$\ell(\gamma)=0\iff$  $\gamma$ is a unit in $\Gamma$ (i.e. $\gamma$ is a constant path).

Let $(X,\Gamma,\ell)$ be a set equipped with a groupoid of paths and with a weak length structure.
We set 
\begin{equation} \label{weakm}
\displaystyle \delta_{\ell}(x,y)=\inf_{\gamma\in\Gamma_{x,y}} \ell(\gamma),
\end{equation}
where 
\[\Gamma_{x,y}=\{\gamma\in \Gamma\ \vert \ \gamma \textrm{ joins $x$ to $y$ }\}.\]
It is easy to see that the function $\delta_{\ell}$ is a weak metric on $X$.

\begin{definition} Let $(X,\Gamma,\ell)$ be a set equipped with a groupoid of paths and with a weak length structure. The weak metric $\delta_{\ell}$ defined in (\ref{weakm}) is called the \emph{weak metric associated to the weak length structure $\ell$}. A \emph{weak length metric space} is a weak metric space $X$ obtained from such a triple $(X,\Gamma,\ell)$ by equipping $X$ with the associated weak metric $\delta_{\ell}$. 
\end{definition}

%_________________________
\section{Weak Finsler structures}

We introduce a general notion of Finsler structure, which we call \emph{weak Finsler structure}, and which can be considered as an infinitesimal notion of weak length structure. 

\begin{definition}
Let $M$ be a $C^1$ manifold and let $TM$ be its tangent bundle. A \emph{weak  Finsler structure} on $M$ is a subset $\widetilde{\Omega}\subset TM$ such that for each $x$ in $M$, the subset $\Omega_x=\widetilde{\Omega}\cap T_xM$ of the tangent space $T_xM$ of $M$ at $x$ is convex and contains the origin.
\end{definition}

We provide the set of all weak Finsler structures on $M$ with the order relation $\preceq$ defined as follows:
$$ 
 \widetilde{\Omega_1} \preceq \widetilde{\Omega_2} \ \Leftrightarrow \
 \widetilde{\Omega_1} \supset  \widetilde{\Omega_2}.
$$

\begin{examples} In  the following examples, $M$ is a $C^1$ manifold.
\begin{enumerate}
\item  $\widetilde{\Omega} =  TM$ is a weak Finsler structure, which we call the \emph{minimal} weak Finsler structure.
\item  $\widetilde{\Omega} =  M \subset TM$, embedded as the zero section, is a weak Finsler structure which we call the \emph{maximal} weak Finsler structure.
\item If $\widetilde{\Omega}$ and $\widetilde{\Omega'}$ are two Finlser structures on $M$, then $\widetilde{\Omega}\cap \widetilde{\Omega'}\subset TM$ is also a Finsler structure. 
\item If $\widetilde{\Omega}$ and $\widetilde{\Omega'}$ are two Finlser structures on $M$, then, taking the union of the Minkowski sums ${\Omega_x}+\Omega'_x$ of the convex sets in each tangent space $T_x M$, we obtain the \emph{Minkowski sum Finsler structure} $\widetilde{\Omega}+ \widetilde{\Omega'}\subset TM$.
\item If $\omega$ is a differential 1-form on $M$, then
\[
  \widetilde{\Omega}_{\omega}= \{(x,\xi)\in TM\ \vert \ \omega_x(\xi)\leq 1\}
\] 
and 
\[
 \widetilde{\Omega}_{\vert\omega\vert}= \{(x,\xi)\in TM\ \vert \ \vert \omega_x\vert (\xi)\leq 1\}
\] 
are weak Finsler structures on $M$.
\item If $\omega$ and $\omega'$ are two 1-froms on $M$, then $\max(\omega,\omega')$ defines a weak  Finlser structure on $M$.
\item If $\widetilde{\Omega}$ is a weak Finlser structure on $M$ and if
$N\subset M$ is a $C^1$ submanifold, then $\widetilde{\Omega}_N= \widetilde{\Omega}\cap TN$ is a weak Finlser structure on $N$, called the weak Finsler structure \emph{induced} by the embedding 
$N \subset M$.

\item  If $\widetilde{\Omega}$ is a weak Finlser structures on $M$, if $N$ is a $C^1$ manifold and if $f:N\to M$ is a $C^1$ map, then $(Tf)^{-1}(\widetilde{\Omega})\subset TN$ is a Finsler structure on $N$. We denote it by $f^*(\widetilde{\Omega})$ and call it the 
\emph{pull back} of $\widetilde{\Omega}$ by the map $f$.
\end{enumerate}
\end{examples}

\begin{definition}[Lagrangian]
The \emph{Lagrangian} of a weak Finlser structure $\widetilde{\Omega}$ on a $C^1$ manifold $M$ is the function on the tangent bundle $TM$ whose restriction to each tangent space $T_x$ is the Minkowski function of $\Omega_x$. It is thus  defined by
\[
 p(x,\xi)=p_{\widetilde{\Omega}}(x,\xi)=\inf\{ t\ \vert \ t^{-1}\xi\in\Omega_x\}.
\]
\end{definition}
The quantity $p(x,\xi)$ is also called the \emph{Finsler norm} of the vector $(x,\xi)$ relative to the given weak Finlser structure.

\begin{example} 
Let $g$ be a Riemannian metric on $M$, let $\omega$ is a differential 1-form and let 
$\mu$   be a smooth function on $M$ satisfying $\vert \mu \omega_x\vert  < 1 $ at every point $x$ in $M$. Then, $p=\sqrt{g} + \mu \omega$  is the Lagrangian of a Finsler structure on $M$. Such a Finlser structure is usually called a \emph{Randers metric} on $M$, and it has applications on physics (cf. e.g. \cite{BCS} \S 11.3, and see also \cite{BRS} for the relation of this metric with the Zermelo navigation problem.)
\end{example}

\begin{lemma} \label{lem:zero}
 Let  $\widetilde{\Omega}$ be a weak Finlser structure on $M$. Assume that $M$ (considered as a subset of $TM$ - the zero section) is contained in the interior of $\widetilde{\Omega} \subset TM$. Then  the associated Lagrangian $p: TM\to\mathbb{R}$ is upper semi-continuous.
\end{lemma}

\begin{proof} 
The hypothesis implies that for every $x$ in $M$, the interior of each convex set $\Omega_x = \widetilde{\Omega}\cap T_xM \subset T_xM$  is nonempty. Therefore, the usual interior and the relative interior of $\Omega_x$ coincide. Property (\ref{pp}) of 
Proposition \ref{prop:min} implies then that the Lagrangian of $\widetilde{\Omega}$ coincides with the Lagrangian of its interior $\mathrm{Int }\left(\widetilde{\Omega}\right)$.

One may therefore assume without loss of generality that  $\widetilde{\Omega} \subset TM$
is an open set, and in particular 
$$
  \widetilde{\Omega} = \{ (x,\xi) \in TM \tq   \ p(x,\xi) < 1\}
$$
(see Proposition \ref{prop:min} (\ref{pro:8})).
Now for  any $t\in \mathbb{R}$, the sublevel set  $\{ p(x,\xi) < t \}$ is either empty (when $t\leq 0$)
or it is  homothetic to the open set $\widetilde{\Omega} \subset TM$ (when $t>0$). In any case, it is  an open 
subset of $TM$, and $p : TM \to \mathbb{R}$ is therefore upper semi-continuous.

\end{proof}

\medskip

\begin{proposition} Let $\widetilde{\Omega}$ be a Finsler structure on a $C^1$ manifold $M$ and let $p_{\widetilde{\Omega}}: TM\to \mathbb{R}$ be the associated Lagrangian. Then, 
\begin{enumerate}
\item for every $x$ in $M$, the function $\xi\mapsto p(x,\cdot)$ is a weak norm on $T_xM$;
\item if $\widetilde{\Omega'}\subset TM$  is another Finsler structure on $M$, with   associated Lagrangian $p_{\widetilde{\Omega'}}$,
then we have the equivalence 
\[\widetilde{\Omega}  \preceq  \widetilde{\Omega'}\iff
p_{\widetilde{\Omega}}\leq p_{\widetilde{\Omega'}},\]
\item \label{item:Borel} $p_{\widetilde{\Omega}}:TM\to \mathbb{R}$ is Borel-measurable.
\end{enumerate}
\end{proposition}

\begin{proof}  The first two assertions are easy to check and we only prove the last one.
If $M$ is contained in the interior of $\widetilde{\Omega} \subset TM$, then, by Lemma \ref{lem:zero},  the Lagrangian $p$ is upper semi-continuous and therefore Borel measurable. In the general case, $M$ is contained in $\widetilde{\Omega}$ but not necessarily in its interior.  We consider a decreasing  sequence  
$$
 TM \preceq \widetilde{\Omega}_1 \preceq \widetilde{\Omega}_2  \preceq \cdots  \preceq \widetilde{\Omega}
$$
of weak Finsler structures such that $M$ is contained in the interior of $\widetilde{\Omega}_j \subset TM$
for every $j\in \mathbb{N}$ and 
$$
 \widetilde{\Omega} = \bigcap_{j=1}^{\infty} \widetilde{\Omega}_j 
$$
We then have $p_{\widetilde{\Omega}_1}\leq p_{\widetilde{\Omega}_2} \leq \cdots \leq  p_{\widetilde{\Omega}}$ 
and
$$
 p_{\widetilde{\Omega}} = \sup_j p_{\widetilde{\Omega}_j} = \lim_{j\to \infty} p_{\widetilde{\Omega}_j}
$$
Therefore
$p_{\widetilde{\Omega}}$ is the limit of a sequence of Borel measurable functions and is thus 
Borel measurable.

\end{proof}

We shall say that the Finlser structure $\widetilde{\Omega}$ is \emph{smooth} if $p$ is smooth.

\begin{definition}[The weak length structure associated to a weak Finsler structure]
Let $M$ be a $C^1$ manifold equipped with a weak Finlser structure $\widetilde{\Omega}$ with Lagrangian $p$. There is an associated weak length structure on $M$, defined by taking $\Gamma$ to be the groupoid of piecewise $C^1$ paths, and defining, for each $\gamma:[a,b]\to M$ in $\Gamma$,
\begin{equation}\label{Lebesgue}
 \ell(\gamma)=\int_a^b p(\gamma(t),\dot{\gamma} (t)) dt.
\end{equation}
\end{definition}

\begin{remark} In Equation (\ref{Lebesgue}), 
$\gamma$ and $\dot{\gamma}$ are continuous, and since $p$ is Borel-measurable, the map $t\mapsto p(\gamma(t),\dot{\gamma} (t))$ is nonnegative  and measurable. Therefore, the Lebesgue integral is well defined.
\end{remark}

%________________________________
\section{The tautological weak Finsler structure}

In this section, $\Omega$ is an open convex subset of $\mathbb{R}^n$. We shall use  the natural identification $T\Omega\simeq \Omega\times \mathbb{R}^n$.

\begin{definition}[The tautological weak Finsler structure]  
 The \emph{tautological weak Finsler structure} on $\Omega$  is the weak Finsler structure $\widetilde{\Omega}\subset T\Omega$ defined by
 \[
 \widetilde{\Omega}=\{(x,\xi)\in\Omega\times \mathbb{R}^n\ \vert \ x\inÊ\Omega\text{ and } x+\xi\in\Omega\}.
 \]
\end{definition}

This structure is called  ``tautological" because the fibre over each point $x$ of $\Omega$ is the set $\Omega$ itself (with the origin at $x$).

The proof of next proposition follows easily from the definitions.
  
\begin{proposition}\label{prop:s} Let $\Omega$ be an open
convex subset of $\mathbb{R}^n$ equipped with its tautological weak Finsler structure 
 $\widetilde{\Omega}$. Then, for every $x$ in $\Omega$, the Finsler norm of any tangent vector $\xi$ at $x$ is given by
$p_{\Omega,x}(\xi)$, where $p_{\Omega,x}$ is the Minkowski function of $\Omega$ with respect to $x$.
\end{proposition}
 
Given an open convex subset $\Omega$ of $\mathbb{R}^n$, we denote by $d_{\Omega}$ the weak length metric associated to the tautological weak Finsler structure on $\Omega$. This weak metric is thus defined by 
\begin{equation}\label{ }
 d_{\Omega}(x,y) = \inf_{\gamma\in\Gamma_{x,y}} \int_{\gamma} p(\gamma(t),\dot{\gamma} (t)) dt.
\end{equation}
where $\Gamma_{x,y}$ is the set of piecewise $C^1$ paths joining $x$ to $y$.

\begin{lemma}\label{lemma:inclusion}
Let $\Omega$ and $\Omega'$ be two convex open subsets of $\mathbb{R}^n$ satisfying $\Omega\subset\Omega'$, then $d_{\Omega'}\leq d_{\Omega}$.
\end{lemma}
 
\medskip

In the rest of this paper, we shall use the following notations:  
For $x$ and $y$ in $\mathbb{R}^n$, we denote by $\vert x-y\vert$ their Euclidean distance. Given two distinct points $x$ and $y$ in $\Omega$, $R(x,y)$ denote the Euclidean ray starting at $x$ and passing through $y$. In the case where $R(x,y)\not\subset\Omega$ we set  $a^+=a^+(x,y)=R(x,y)\cap\partial\Omega$.

\begin{theorem}\label{prop:th}
Let $\Omega$ be an open convex subset of $\mathbb{R}^n$ equipped with its tautological weak Finsler structure. Then, for every $x$ and $y$ in $\Omega$, the Euclidean segment connecting $x$ and $y$ is of minimal length, and the associated weak metric on $\Omega$ is given by 
 \[
 \displaystyle d_{\Omega}(x,y)=
\begin{cases} \displaystyle \log \frac{\vert x-a^+\vert}{\vert y-a^+\vert} & \text{ if } x\not= y \text{ and } R(x,y)\not\subset \Omega\\
0 & \text{ otherwise}.
\end{cases}
\]
\end{theorem}

\begin{proof}
As before, we let $d_{\Omega}$ denote the weak metric defined by the tautological weak Finsler structure on $\Omega$. We also denote by $\ell(\gamma)$ the length of a path $\gamma$ for the tautological weak Finsler weak length structure.

The proof of the theorem is done in four steps. 

 \medskip

\emph{Step 1.---}
Suppose that $R(x,y)\subset\Omega$.  Consider the linear path $\gamma:[0, \vert x-y\vert] \to \Omega$ defined by

\begin{equation}\label{eq:gamma}
\gamma(t)=x+t\frac{y-x}{\vert y-x\vert}.
\end{equation}

The derivative of the path $\gamma$ is the constant vector 
\[\dot{\gamma}(t)=\frac{y-x}{\vert y-x\vert}.\]
Therefore, $\displaystyle p_{\widetilde{\Omega}}(\gamma(t),\dot{\gamma}(t))
=\frac{1}{\vert y-x\vert}p_{\widetilde{\Omega}}(\gamma(t),y-x)$, which is equal to 0 since $R(x,y)\subset\Omega$.

Now the path $\gamma$ has length zero and satisfies $\gamma(0)=x$ and $\gamma(\vert y-x\vert)=y$. Therefore
$d_{\Omega}(x,y)= 0$.

 \medskip

In the rest of this proof, we suppose that $R(x,y)\not\subset\Omega$.

\emph{Step 2.---} We show that for every distinct points $x$ and $y$ in $\Omega$ and for every Euclidean segment $\gamma$ joining $x$ to $y$, we have
\begin{equation}\label{eq:length}
\displaystyle d_{\Omega}(x,y)\leq \ell(\gamma)=\log\frac{\vert x-a^+\vert }{\vert y-a^+\vert}.
\end{equation}

Using the radial function $r_{\Omega,x}$ introduced in \S\ref{s:1}, we can write
\[a^+=a(x,y-x)=x+r_{\Omega,x}(y-x)\cdot (y-x).\] 
To compute the Finsler length of the Euclidean segment $[x,y]$, we parametrize  it as  the path $\gamma$ defined in (\ref{eq:gamma}).

For $0\leq t\leq \vert x-y\vert$, let $r(t)= \vert x-\gamma(t)\vert$. Then, $r(t)=r_{\Omega,x}(\gamma(t),\dot{\gamma}(t))$, and it is easy to see that 
\[r(t)=\vert x-a^+\vert -t.\]

Then, we have $r'(t) = -1$ and therefore
\[
\ell(\gamma) =\int_0^{\vert y-x\vert}\frac{dt}{r(t)}=-\int_0^{\vert y-x\vert}\frac{r'(t)dt}{r(t)}
 =-\log\big(r(t)\big)\Big\vert_{t=0}^{t=\vert y-x\vert}
=\log\frac{\vert x-a^+\vert}{\vert y-a^+\vert}.
\]
This gives the desired inequality (\ref{eq:length}).
 
\medskip

\emph{Step 3.---} We complete the proof of the theorem in the particular case where $\Omega$ is a half-space.  By the invariance of the tautological Finsler structure under the group of affine transformations,  it suffices to consider the case where $\Omega$ is the
half-space $H\subset \mathbb{R}^n$  defined by the equation
\[H = \{ x \in \mathbb{R}^n \tq \langle \nu , x \rangle  \leq s\},\]
for some vector $\nu$ in $\mathbb{R}^n$ (which is orthogonal to the hyperplane bounding $H$) and for some $s$ in $\mathbb{R}$. Recall that the Minkowski function associated to $H$ is given by the formula   
\[
 p_{H}(x,\xi) = 
  \max\left\{ \frac{\langle \nu ,\xi  \rangle}{s-\langle \nu , x \rangle},0\right\}.
\]

Consider now an arbitrary piecewise $C^1$ path $\alpha:[0,1]\to H$ such that  $x=\alpha(0)$ and $y=\alpha(1)$. Then, 
\[
\ell(\alpha)=\int_0^1\max\left\{ \frac{\langle \nu,\dot{\alpha}(t)\rangle}{s-\langle \nu,\alpha(t)\rangle},0 \right\} dt
\geq \int_0^1\frac{\langle \nu,\dot{\alpha}(t)\rangle}{s-\langle \nu,\alpha(t)\rangle}dt.
\]
We have
\[\frac{\langle \nu,\dot{\alpha}(t)\rangle}{s-\langle \nu,\alpha(t)\rangle}=-\frac{d}{dt}\big(\log \big( s-\langle \nu,\alpha(t)\rangle\big)\big).\]
Therefore,
\[\ell(\alpha)\geq -\log\big( s-\langle \nu,\alpha(1)\rangle\big) + \log\big( s-\langle \nu,\alpha(0)\rangle\big)= \log\frac{s-\langle\nu,x\rangle}{s-\langle \nu,y\rangle}.\]

Now we note that 
\[s-\langle \nu,x\rangle=s-\langle \nu,x-a^+\rangle-\langle \nu,a^+\rangle=\langle x-a^+,-\nu\rangle=\langle\nu,a^+-x\rangle.\]
Likewise, 
\[s-\langle \nu,y\rangle= \langle\nu,a-y\rangle.\]
Thus, we obtain
\[\ell(\alpha)\geq \log\frac{\langle\nu,a^+-x\rangle}{\langle\nu,a^+-y\rangle}.\]
Now using the fact that the three points $x,y,a^+$ are aligned in that order and that $\nu$ is not parallel to the vector $x-y$, we easily see that
\[\frac{\langle\nu,a-x\rangle}{\langle\nu,a-y\rangle}=\frac{\vert x-a^+\vert}{\vert y-a^+\vert},\]
which gives 
\[\ell(\alpha)\geq \log\frac{\vert x-a^+\vert }{\vert y-a^+\vert }.\]
Since $\alpha$  was arbitrary, we have
\[ d_H(x,y)\geq \log\frac{\vert x-a^+\vert }{\vert y-a^+\vert }.\]
Combining this inequality and the inequality (\ref{eq:length}), we obtain, in the case where $\Omega =H$ is a half-space,
\[
 d_{H}(x,y)= \log\frac{\vert x-a^+\vert }{\vert y-a^+\vert}.
\]
In particular any Euclidean segment is length minimizing.

\medskip

\emph{Step 4.---} Now we prove the proposition for a general open convex set $\Omega$.

Let $x$ and $y$ be two elements in $\Omega$ and consider the Euclidean ray $R(x,y)$. 

By hypothesis, we have $L\not\subset \Omega$, and as before, we set $a^+=R(x,y)\cap \partial\Omega$. 
We let $A$ denote a support hyperplane to $\Omega$ through $a^+$, and we let $H$ be the open half-space containing $\Omega$ and whose boundary is equal to $A$.
Using Lemma \ref{lemma:inclusion} and  Step 3, we have 
\[
d_{\Omega}(x,y)\geq d_H(x,y)=\log \frac{\vert x-a^+\vert }{\vert y-a^+\vert }.
\]
Combining this with  the inequality (\ref{eq:length}) we obtain $d_{\Omega}(x,y)=\log \frac{\vert x-a^+\vert }{\vert y-a^+\vert }$. The argument also proves that any Euclidean segment $\gamma$ is length minimizing.
This completes the proof of Theorem \ref{prop:th}.

\end{proof}

\section{The Funk weak metric}
   
In this and the following section, we give a quick overview of the Funk weak metric, of its geodesics, of its balls and of its topology. 

The Funk weak metric is a nice example of a weak metric, and a geometric study of this weak metric is something which seems missing in the literature. We study this weak metric in more detail in \cite{PT2}.
   
In this section, $\Omega$ is a nonempty open convex subset of $\mathbb{R}^n$. 
We use the notations $a^+$, $R(x,y)$, etc. established in the preceding section.

\begin{definition}[The Funk weak metric]\label{def:Funk}
 The \emph{Funk weak metric} of $\Omega$, denoted by $F_{\Omega}$,  is defined, for $x$ and $y$ in $\Omega$, by the formula
 \[
  \displaystyle F_{\Omega}(x,y)=
\begin{cases} \displaystyle \log \frac{\vert x-a^+\vert}{\vert y-a^+\vert} & \text{ if } x\not= y \text{ and } R(x,y)\not\subset \Omega\\
0 & \text{ otherwise}.
\end{cases}
\]
\end{definition}

Observe that  Theorem \ref{prop:th} says that the Funk weak metric is the weak metric associated to the tautological Finsler structure
in $\Omega$. In particular the triangle inequality is verified. Another proof of the  triangle inequality is given in \cite{Zaustinsky} p. 85. This proof is not trivial and uses arguments similar to those of the classical proof of the triangle inequality for the Hilbert metric, as given by D. Hilbert in \cite{Hilbert2}. 

If $\Omega=\mathbb{R}^n$, then $F\equiv 0$. We shall henceforth assume that  $\Omega\not=\mathbb{R}^n$ whenever we shall deal with the Funk weak metric of an nonempty open convex subset $\Omega$ of $\mathbb{R}^n$.

The Funk weak metric is always unbounded. Indeed, if $x$ is any point in $\Omega$ and if $x_n$ is any sequence of  points in that space converging to a point on $\partial\Omega$ (convergence here is with respect to the Euclidean metric), then $F_{\Omega}(x,x_n)\to \infty$. Notice that on the other hand $F_{\Omega}(x_n,x)$ is bounded.

\begin{example}[The upper half-plane]
Let $\Omega=H\subset\mathbb{R}^2$ be the upper half-plane, that is,
\[H=\{(x_1,x_2)\in\mathbb{R}^2\ \vert \ x_2>0\}.\]
Then, for $x=(x_1,x_2)$ and $y=(y_1,y_2)$ in $H$, we have
\[F_H(x,y)=\max \left\{\log\frac{x_2}{y_2},0\right\}.\]
\end{example}    

The following three propositions are easy consequences of the definitions and they will be used below. We take $\Omega$ to be again a nonempty open subset of $\mathbb{R}^n$.

\begin{proposition}
Let $\Omega'\subset\Omega$ be the intersection of $\Omega$ with an affine subspace of $\mathbb{R}^n$, and suppose that $\Omega'\not=\emptyset$. Then, $F_{\Omega'}$ is the weak metric induced by $F_{\Omega}$ on $\Omega'$.
\end{proposition}

\begin{proposition}\label{eq:conv-F}
 In the case where $\Omega$ is bounded, the Funk weak metric $F_{\Omega}$ is strongly separating, and we have the following equivalences:
\begin{equation}\label{eq.conv2}
F_{\Omega}(x,x_n)\to 0\iff F_{\Omega}(x_n,x)\to 0\iff\vert x-x_n\vert \to 0.
\end{equation}
 \end{proposition}

\begin{proposition}
Let $\Omega_1$ and $\Omega_2$ be two open convex subsets of $\mathbb{R}^n$. Then, 
\[
 F_{\Omega_{1}\cap \Omega_{2}}=\max \left\{ F_{\Omega_{1}}, F_{\Omega_{1}}\right\}.
\]
\end{proposition}

%_______________________
\section{On the geometry of the Funk weak metric}
 
In this section, we study the geodesics, and then, the geometric balls of the Funk weak metric. 

\begin{proposition}\label{aligned-Funk}
Let $x$, $y$ and $z$ be three points in $\Omega$ lying in that order on a Euclidean line. Then, we have
$F(x,y)+F(y,z)=F(x,z)$.
\end{proposition}

This results follows from Theorem \ref{prop:th}, but it is also quite simple to prove it directly.
\begin{proof}
We can assume that the three points are distinct, otherwise the proof is trivial. We have $R(x,y)\subset\Omega\iff R(x,z)\subset\Omega\iff R(y,z)\subset\Omega$, and this holds if and only if the three quantities $F(x,y)$, $F(y,z)$ and $F(x,z)$ are equal to 0. Thus, the conclusion also holds trivially in this case. Therefore, we can assume that $R(x,y)\not\subset\Omega$. In this case, we have $a^+(x,y)= a^+(x,z)=a^+(y,z)$. Denoting this common point by $a^+$, we have
\[\frac{\vert x-a^+\vert}{\vert y-a^+\vert} \frac{\vert y-a^+\vert}{\vert z-a^+\vert}=
\frac{\vert x-a^+\vert}{\vert z-a^+\vert},\]
which implies 
\[\log \frac{\vert x-a^+\vert}{\vert y-a^+\vert}+\log \frac{\vert y-a^+\vert}{\vert z-a^+\vert}=
\log\frac{\vert x-a^+\vert}{\vert z-a^+\vert},\]
which completes the proof.
\end{proof}

\begin{corollary}\label{co:Funk-geodesic}
The Euclidean segments in $\Omega$ are geodesic segments for the Funk weak metric on $\Omega$. 
\end{corollary}

 Since the open set $\Omega$ is convex,  Corollary \ref{co:Funk-geodesic} implies that $(\Omega,F_{\Omega})$ is a geodesic weak metric space (any two points can be joined by a geodesic segment). It also says that $(\Omega,F_{\Omega})$ is a Desarguesian space in the sense of H.  Busemann (see \cite{Busemann1955}).

Notice that in general, the Euclidean segments are not the only geodesic segments for a Funk weak metric. In fact, the following proposition implies that there exist other types of geodesic segments in $\Omega$, provided there exists a Euclidean segment of nonempty interior contained in the boundary of $\Omega$.

\begin{proposition}\label{prop:obtuse}
Let $\Omega$ be an open convex subset of $\mathbb{R}^n$ such that $\partial\Omega$ contains a Euclidean segment $[p,q]$ and let $x$ and $z$ be two points in $\Omega$ such that  $R(x,z)\cap[p,q]\not=\emptyset$. Let $\Omega'$ be the intersection of $\Omega$ with the affine subspace of $\mathbb{R}^n$ spanned by $\{x\}\cup [p,q]$. Then, for any point $y$ in $\Omega'$ satisfying 
 $R(x,y)\cap[p,q]\not=\emptyset$ and  $R(y,z)\cap[p,q]\not=\emptyset$,
 we have $F(x,y)+F(y,z)=F(x,z)$.
\end{proposition}

\begin{figure}[!htbp] 
\centering

\psfrag{p}{\small $p$}
\psfrag{q}{\small $q$}
\psfrag{x}{\small $x$}
\psfrag{y}{\small $y$}
\psfrag{z}{\small $z$}
\psfrag{a}{\small $x'$}
\psfrag{b}{\small $y'$}
\psfrag{c}{\small $z'$}
\includegraphics[width=.50\linewidth]{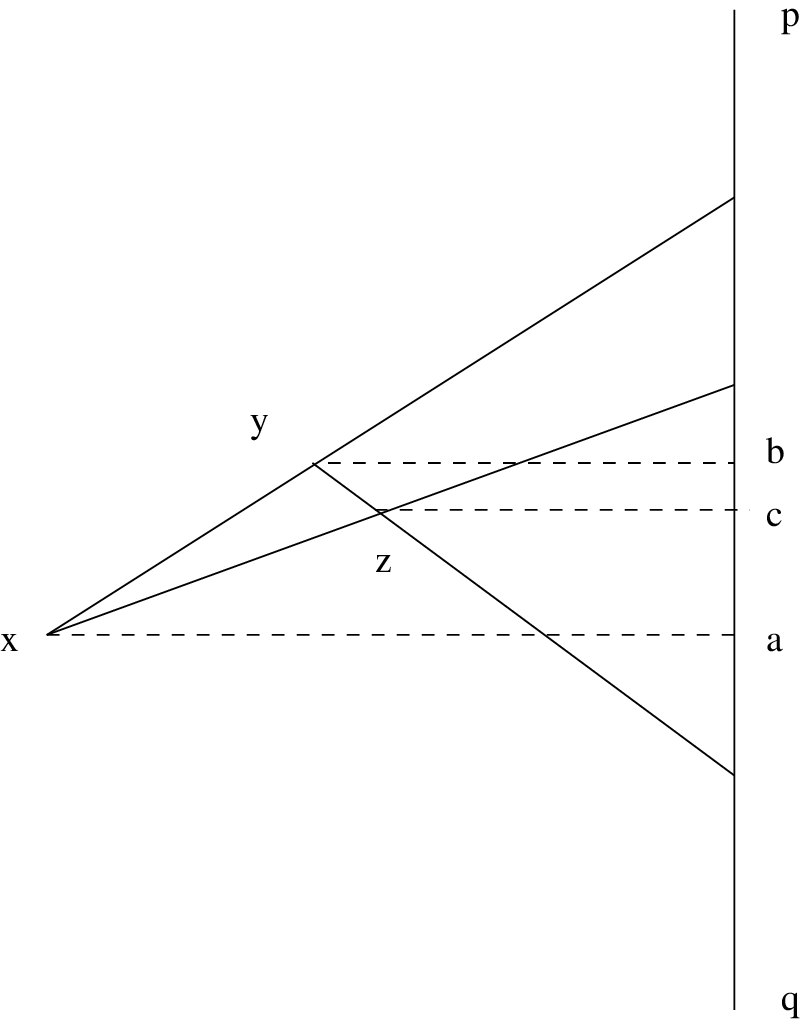}
\caption{\small}
\label{fig:finsler1}
\end{figure}

 \begin{proof} It suffices to work in the space $\Omega'$. Let $x'$, $y'$ and $z'$ denote the feet of the perpendiculars from $x$ and $z$ respectively on the Euclidean line joining the points $p$ and $q$ (see Figure \ref{fig:finsler1}). 
Let $b = R(x,z)\cap[p,q]$. Since the triangles  $bxx'$  and  $bzz'$
are similar, we have  
\[
F(x,z)=\log\frac{\vert x-b\vert}{\vert z-b\vert}
  =\log\frac{\vert x-x'\vert}{\vert z-z'\vert}.
\]
Similar formulas hold for $F(x,y)$ and $F(y,z)$. Therefore,
\begin{align*}
F(x,z) & = \log\frac{\vert x-x'\vert}{\vert z-z'\vert}
     \\ & =
     \log\left(\frac{\vert x-x'\vert}{\vert y-y'\vert}\frac{\vert y-y'\vert}{\vert z-z'\vert}\right) 
      \\ & = \log\left(\frac{\vert x-x'\vert}{\vert y-y'\vert}\right)  + 
    \log\left(\frac{\vert y-y'\vert}{\vert z-z'\vert} \right)
    \\ & = F(x,y) + F(y,z).
\end{align*}

\end{proof}

\begin{remark}
By taking limits of polygonal paths, we can easily construct, from Proposition \ref{prop:obtuse}, smooth paths which are not Euclidean paths and which are geodesic  for the Funk weak metric.
\end{remark}

  \bigskip

\begin{proposition}\label{prop:geod-Funk}
  Let $\Omega$ be an open convex subset of $\mathbb{R}^n$. Let $x$ and $z$ be two distinct points in $\Omega$ such that $R(x,z)\cap\partial\Omega\not=\emptyset$ and such that at the point $b=R(x,z)\cap\partial\Omega$,  there is a support hyperplane whose intersection with $\partial\Omega$ is reduced to $b$. Let $y$ be a point in $\Omega$ such that the three points $x,y,z$ in $\Omega$ do not lie on the same affine line. Then, $F(x,z)<F(x,y)+F(y,z)$.
  \end{proposition}

        \begin{proof}
To prove the proposition, we work in the affine plane spanned by $x$, $y$ and $z$ and therefore we can assume without loss of generality that $n=2$. 
  
 We assume that the intersection points of $R(x,y)$ and $R(y,z)$ with $\partial\Omega$ are not empty, and we let $a$ and $c$ be respectively these points. From the hypothesis, there is a support line of $\Omega$ (which we call $D$) at $b$ whose intersection with $\partial\Omega$ is reduced to the point $b$.
  
For the proof, we distinguish three cases.

  \begin{figure}[!htbp]
\centering

\psfrag{D}{\small $D$}
\psfrag{m}{\small $a'$}
\psfrag{a}{\small $a$}
\psfrag{x}{\small $x$}
\psfrag{y}{\small $y$}
\psfrag{z}{\small $z$}
\psfrag{c}{\small $c$}
\psfrag{b}{\small $b$}
\psfrag{n}{\small $c'$}
\includegraphics[width=.25\linewidth]{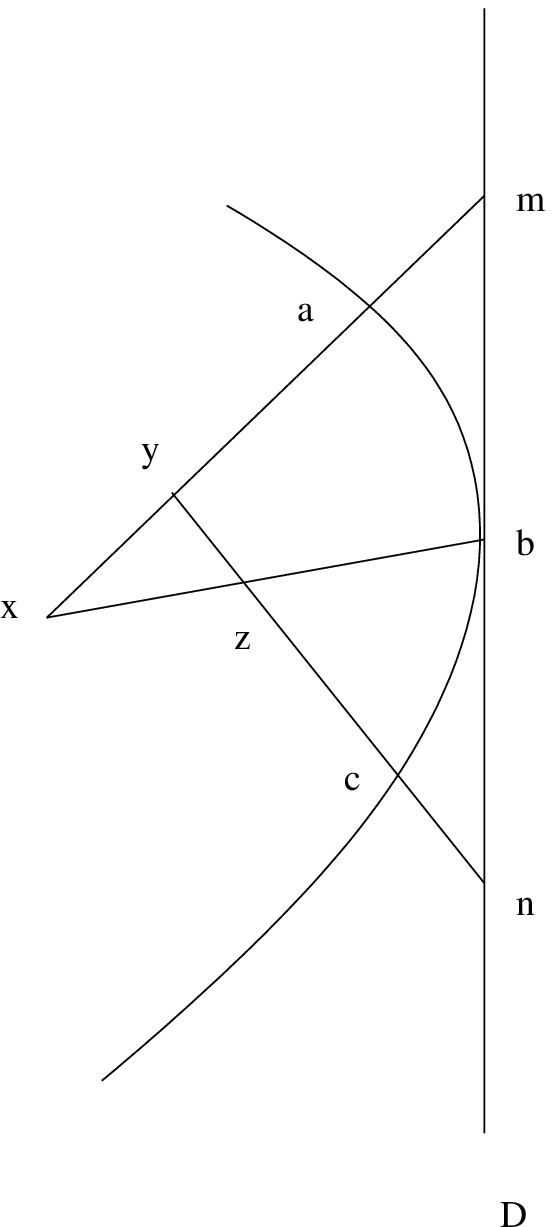}
\caption{\small {}}
\label{fig:finsler2}
\end{figure}

\noindent {Case 1.---} The two rays $R(x,y)$ and $R(y,z)$ intersect the line $D$ (see Figure \ref{fig:finsler2}).

Let   $a'$ and $c'$ be respectively these intersection points. Note that the three points $a'$, $b$ and $c'$ are in that order on $D$.    
By reasoning with projections on the line $D$   and arguing as we did in the proof of
Proposition \ref{prop:obtuse}, we have
  \[\frac{\vert x-b\vert}{\vert z-b\vert}= \frac{\vert x-a'\vert}{\vert y-a'\vert}
  \frac{\vert y-c'\vert}{\vert z-c'\vert}.\]
  Since we have 
  \[\frac{\vert x-a'\vert}{\vert y-a'\vert}< \frac{\vert x-a\vert}{\vert y-a\vert}
  \]
and 
   \[\frac{\vert y-c'\vert}{\vert z-c'\vert}< \frac{\vert y-c\vert}{\vert z-c\vert},
  \]
  we obtain
    \[\frac{\vert x-b\vert}{\vert z-b\vert}<\frac{\vert x-a\vert}{\vert y-a\vert}
  \frac{\vert y-c\vert}{\vert z-c\vert}\]
   which gives, by taking logarithms,  $ F(x,z)<F(x,y)+F(y,z)$.
 
 \medskip
 
 \noindent {Case 2.---} The ray $R(x,y)$ intersects $D$ and the ray $R(y,z)$ does not intersect $D$ (Figure \ref{fig:finsler3}). We let as before $a'$ denote the point $R(x,y)\cap D$.
 
 Let $D'$ be the Euclidean line passing through $z$ and parallel to $D$. The hypotheses in the case considered imply that the line $D'$ intersects the segment $[x,y]$. Let $y'$ be this intersection point. The point $y'$ is contained in $\Omega$.
 
 We have, as in Case 1, 
 \[\displaystyle F(x,z)=\log \frac{\vert x-b\vert}{\vert z-b\vert}\]
  and 
 \[F(x,y)=\log \frac{\vert x-a\vert}{\vert y-a\vert}> \log \frac{\vert x-a'\vert}{\vert y-a'\vert}.\]
 Now we have
 \[
 \frac{\vert x-b\vert}{\vert z-b\vert}=\frac{\vert x-a'\vert}{\vert y'-a'\vert}<
 \frac{\vert x-a'\vert}{\vert y-a'\vert},
 \]
 that is, $F(x,z)<F(x,y)$, which implies the desired result.

  \begin{figure}[!htbp]
\centering

\psfrag{D}{\small $D$}
\psfrag{L}{\small $D'$}
\psfrag{m}{\small $a'$}
\psfrag{k}{\small $x$}
\psfrag{y}{\small $y$}
\psfrag{z}{\small $z$}
\psfrag{b}{\small $b$}
\psfrag{s}{\small $y'$}
\includegraphics[width=.30\linewidth]{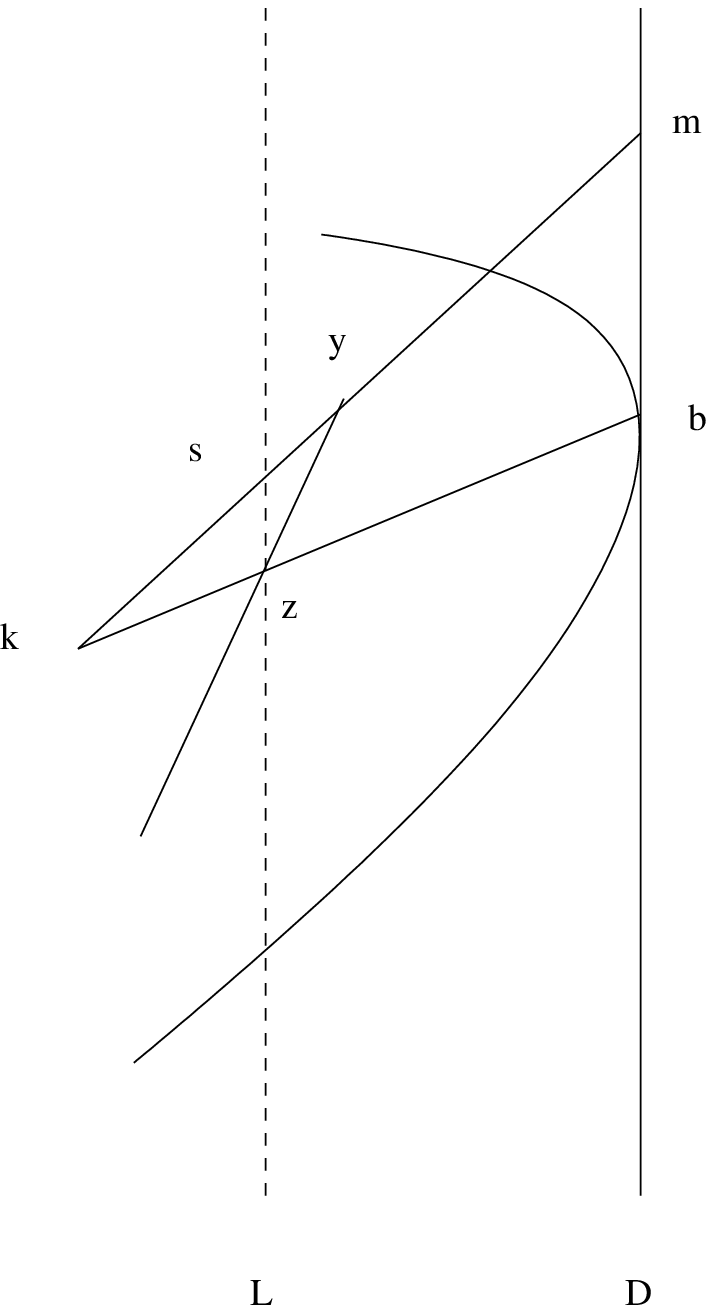}
\caption{\small }
\label{fig:finsler3}
\end{figure}

 \medskip
 
 \noindent {Case 3.---} The ray $R(x,y)$ does not intersect the line $D$. This case can be treated as Case 2, and we have in this case $F(x,y)<F(y,z)$, which implies the desired result.
 
\end{proof}

    The following is a direct consequence of Proposition \ref{prop:geod-Funk}.

  \begin{corollary}\label{cor:geod-Funk}
  Let $\Omega$ be an open bounded strictly convex subset of $\mathbb{R}^n$ and let $x$, $y$ and $z$ be three points in $\Omega$ that are not contained in an affine segment. Then, $F(x,z)<F(x,y)+F(y,z)$.
  \end{corollary}

  \begin{corollary}\label{cor:geo-f}
  Let $\Omega$ be an open bounded strictly convex subset of $\mathbb{R}^n$. Then, the affine segments in $\Omega$ are the only geodesic segments for the Funk weak metric of $\Omega$.
\end{corollary}

\begin{proof} This follows from the previous Corollary and  Corollary \ref{co:Funk-geodesic}, which says that the affine segments are geodesic segments for the Funk weak metric.  
  
\end{proof}

        We recall that a subset $Y$ in  a (weak) metric space $X$ is said to be \emph{geodesically convex}  if for any two points $x$ and $y$ in $Y$, any geodesic segment in $X$ joining $x$ and $y$ is contained in $Y$.

   \begin{corollary}\label{cor:conv-F} 
  Let $\Omega$ be an open bounded strictly convex subset of $\mathbb{R}^n$ and let $\Omega'$ be a subset of $\Omega$. Then, $\Omega'$ is convex with respect to the affine structure of $\mathbb{R}^n$ if and only if $\Omega'$ is a geodesically convex subset of $\Omega$ with respect to the Funk metric $F_{\Omega}$.
    \end{corollary}

    \begin{remark}
Note the formal analogy between Corollary \ref{cor:geo-f} and the following well known result on the geodesic segments of a Minkowski metric on $\mathbb{R}^n$: if the unit ball of a Minkowski  metric is strictly convex, then the only geodesic segments of this metric are the affine segments.
\end{remark}

We now consider spheres and balls in a Funk weak metric space $(\Omega,F)$. As this weak metric is non-symmetric, we have to distinguish between right and left spheres, and we use the following notations.
For any point $x$ in $\Omega$ and any nonnegative real number $\delta$, we set 

 \medskip  \hspace{+.2cm}  $\circ$ \
$B(x,\delta)= \{y\in \Omega\ \vert \ F_{\Omega}(x,y)<\delta\}$ (the \emph{right open ball of center $x$ and radius $\delta$});

 \medskip  \hspace{+.2cm}  $\circ$ \
$B'(x,\delta)= \{y\in \Omega\ \vert \ F_{\Omega}(y,x)<\delta\}$ (the \emph{left open ball of center $x$ and radius $\delta$});

 \medskip  \hspace{+.2cm}  $\circ$ \
$S(x,\delta)= \{y\in \Omega\ \vert \ F_{\Omega}(x,y)=\delta\}$
(the \emph{right sphere of center $x$ and radius $\delta$}); 

 \medskip  \hspace{+.2cm}  $\circ$ \
$S'(x,\delta)= \{y\in \Omega\ \vert \ F_{\Omega}(y,x)=\delta\}$
 (the \emph{left sphere of center $x$ and radius $\delta$}).

  In \cite{Busemann1944} p. 20,  H. Busemann discusses topologies for general weak metric spaces. In the case of a genuine metric space, the open balls are used to define the topology of that space. In general, the collections of left and of right open balls in a weak metric space generate two different topologies. For the Funk weak metric, we have the following
 
If  $\Omega$ is a bounded convex open set of $\mathbb{R}^n$ equipped with its Funk weak metric; then, the collections of left and of right open balls are sub-bases of the same topology on $\Omega$, and this topology coincides with the topology induced from the inclusion of $\Omega$ in $\mathbb{R}^n$.

In the case where the convex open set $\Omega$ is unbounded, the left and the right open balls of the Funk weak metric are always noncompact. In the next proposition, we study these balls in the case where $\Omega$ is bounded.  We recall that a  convex subset of $\mathbb{R}^n$ is unbounded if and only if it contains a Euclidean ray.

 \begin{proposition}\label{prop:balls-homothetic}
 Let $\Omega$ be a bounded convex open subset of $\mathbb{R}^n$,  let $x$ be a point in $\Omega$ and let $\delta$ be a nonnegative real number. Then,
 \begin{enumerate}
\item \label{im1} The right sphere $S(x,\delta)$ is convex as a subset of $\mathbb{R}^n$, and it is compact. Furthermore, this sphere is the image of $\partial\Omega$ by the Euclidean homothety $\sigma$ of center $x$ and factor $(1-e^{-\delta})$.

\item \label{im2} The left sphere $S'(x,\delta)$ is convex as a subset of $\mathbb{R}^n$, and it is equal to the intersection with $\Omega$ of the  image of $\partial\Omega$ by the Euclidean homothety of center $x$ and of factor $(e^{\delta}-1)$, followed by the Euclidean central symmetry of center $x$. The sphere $S'(x,\delta)$ is not necessarily compact.
\end{enumerate}
 \end{proposition}
   \begin{proof}
Let $y$ be a point in $\Omega$ and let us set, as before,  $a^+=R(x,y)\cap \partial\Omega$. We have the following equivalences:
\[y\in S(x,\delta)\iff \log\frac{\vert x-a^+\vert}{\vert y-a^+\vert}=\delta\iff \displaystyle \frac{\vert x-a^+\vert}{\vert y-a^+\vert}=e^{\delta},\] which is easily seen to be equivalent to 
$\vert y-x\vert=\vert x-a^+\vert (1-e^{-\delta})$. From this fact  Property (\ref{im1}) follows easily.

To prove  Property (\ref{im2}), let $a^-=R(y,x)\cap\partial \Omega$. We have the following equivalences:
 \[\log\frac{\vert y-a^-\vert}{\vert x-a^-\vert}=\delta\iff
 \vert y-a^-\vert=e^{\delta}\vert x-a^-\vert,\]
 which is also equivalent to
 \[\vert y-x\vert=(e^{\delta}-1)\vert x-a^-\vert.\]
 
 Thus, $y\in S'(x,\delta)$ if and only if $y$ is in the intersection of $\Omega$ with the image $\sigma(\partial\Omega)$ of $\partial\Omega$ by the Euclidean homothety with center $x$ and of factor $(e^{\delta}-1)$, followed by the Euclidean central symmetry of center $x$. This intersection is convex as a subset of $\mathbb{R}^n$ but it is not necessarily a compact subset of $(\Omega,F)$. Thus, $S'(x,\delta)$ is compact if and only if $\sigma(\partial\Omega)$ is contained in $\Omega$.
 \end{proof}
 
 We note the following ``local-implies-global" property of Funk weak metrics. The meaning of the statement is clear, and it follows directly from Proposition \ref{prop:balls-homothetic} (\ref{im1}).
 
 \begin{corollary} We can reconstruct the boundary $\partial\Omega$ of $\Omega$ from the local geometry at any point of $\Omega$.
\end{corollary}
   
      \begin{corollary}\label{cor:bal-F} Let $\Omega$ be a bounded open strictly convex subset of $\mathbb{R}^n$. Then, the left and right open balls of $\Omega$ are geodesically convex with respect to the Funk weak metric $F_{\Omega}$.
    \end{corollary}

  \begin{proof}
This follows from Proposition \ref{prop:balls-homothetic} and from Corollary \ref{cor:conv-F}.
\end{proof}

 We also deduce from Proposition \ref{prop:balls-homothetic} that for any $x$ and $x'$ in $\Omega$ and for any two positive real numbers $\delta$ and $\delta'$, the right spheres $S(x,\delta)$ and $S(x',\delta')$ are homothetic. 
 
 Thus, for instance, if $\Omega$ is the interior of a Euclidean sphere (respectively, of a Euclidean  ellipsoid) in $\mathbb{R}^n$, then any right sphere $S(x,\delta)$ is a Euclidean sphere (respectively, an ellipsoid).

Note that the proof of  Proposition \ref{prop:balls-homothetic} shows that for a fixed $x$, any two right spheres $S(x,\delta)$ and $S(x,\delta')$ are homothetic by a Euclidean homothety of center $x$, but that in general, a homothety which sends a sphere $S(x,\delta)$ to a sphere $S(x',\delta')$ does not necessarily send the center $x$ of $S(x,\delta)$ to the center $x'$ of $S(x',\delta')$. One can see this fact on the following example: Let $\Omega$ be an open Euclidean disk in $\mathbb{R}^n$, and let us take $x$ to be the Euclidean center of that disk. Then, by symmetry, for any $\delta >0$, the right sphere $S(x,\delta)$ is a Euclidean sphere whose Euclidean and whose metric centers are both at $x$. Now let $x'$ be a point which is close to the boundary of $\Omega$. Obviously, the Euclidean homothety that sends $\partial\Omega$ to $S(x',\delta)$ does not send the center of $\partial\Omega$ to  the (Funk-)geometric center of the sphere $S(x',\delta)$ (recall that the center of this homothety is the point $x$). Now taking a composition of two homotheties, we obtain a Euclidean homothety that sends the geometric sphere $S(x,\delta)$ to the geometric sphere $S(x',\delta)$, and that does not preserve the geometric centers of these spheres.

 \begin{remark} The property for a weak metric on a subset $\Omega$ of $\mathbb{R}^n$ that all the right spheres are homothetic is also shared by the metrics induced by Minkowski weak metrics on $\mathbb{R}^n$.
 \end{remark}

%___________

\end{document}